\documentclass{amsart}
\usepackage[margin=2cm]{geometry}
\usepackage{color,xcolor}
\usepackage{amsfonts,amsmath,amsthm,amssymb}

\usepackage{tikz}
\usepackage[colorlinks=false,pagebackref,hyperindex]{hyperref}

\hypersetup{
	colorlinks,
	citecolor=violet,
	linkcolor=violet,
	urlcolor=blue}

\newcommand{\RR}{\mathbf R}

\newcommand{\NN}{\mathbf N}
\renewcommand{\AA}{\mathbf A}

\theoremstyle{Theorem}
\newtheorem{thm}{Theorem}[section]

\newtheorem*{mainthm}{Theorem}
\newtheorem{cor}[thm]{Corollary}
\newtheorem{lem}[thm]{Lemma}

\theoremstyle{definition}

\newtheorem{dff}[thm]{Definition}
\newtheorem{xmp}[thm]{Example}
\newtheorem{rmk}[thm]{Remark}
\newtheorem{que}[thm]{Question}

\def\HK{\textnormal{HK}}

\def\fm{\mathfrak{m}}
\def\fn{\mathfrak{n}}

\def\cH{\mathcal{H}}

\newcommand\sIJe[1]{I^{\lceil sp^{#1} \rceil} + J^{[p^{#1}]}}

\newcommand{\vol}{\operatorname{vol}}

\newcommand{\Exp}{\operatorname{Exp}}
\newcommand{\Hull}{\operatorname{Hull}}

\colorlet{DG}{green!50!black}
\colorlet{DB}{blue!50!black}

\title{On lower bounds for $s$-multiplicities}
\author{Lance Edward Miller and William D.\ Taylor}
\begin{document}

\maketitle 

\begin{abstract}
A recent continuous family of multiplicity functions on local rings was introduced by Taylor interpolating between Hilbert-Samuel and Hilbert-Kunz multiplicities. The obvious goal is to use this as a  tool for deforming results from one to the other. The values in this family which do not match these classic variants however are not known yet to be well-behaved. This article explores lower bounds for these intermediate multiplicities as well as gives evidence for analogies of the Watanabe-Yoshida minimality conjectures for unmixed singular rings. 
\end{abstract}

\section{Introduction}

The Hilbert-Kunz multiplicity is among the most useful, subtle, and fascinating invariants of positive equicharacterisitc local rings. Since its discovery by Monsky, it has been a focal point of positive characteristic commutative algebra and algebraic geometry. It is well known to stand in stark contrast to the better behaved Hilbert-Samuel multiplicity. Recently, Taylor introduced a continuous family of multiplicities $e_s(I,J)$ for pairs of ideals $I$ and $J$ in such local rings, primary to the unique maximal ideal of a local ring, where $s > 0$ is any positive real number. For large $s$-values, this multiplicity agrees with the Hilbert-Kunz multiplicity of $J$ and for small $s$-values it agrees with the Hilbert-Samuel multiplicity of $I$. This makes it a tantalizing tool to deform between properties of Hilbert-Samuel and Hilbert-Kunz multiplicities. However, the values of $e_s(I,J)$ which lie in between these extremal behaviors are not at all well understood, even when $I = J$. For example, it is not even known that $e_s(\fm,\fm)$ is bounded in terms of the Hilbert-Samuel or Hilbert-Kunz multiplicities of the positive characteristic local ring $(R,\fm,k)$, or even that the function $s \mapsto e_s(\fm,\fm)$ is decreasing. Thus, expected theorems, such as the regularity criteria of Nagata or Watanabe-Yoshida are not present for all $s$-values.  

The purpose of this article is to establish bounds for $s$-multiplicities. We hope that such estimates will provide essential tools for future work and will at least illuminate the behavior of such multiplicities in the cases where they do not recover the more familiar Hilbert-Kunz or Hilbert-Samuel cases. We call these multiplicities {\it intermediate}. One might expect for each $\fm$-primary ideal $I$, $e_s(I,I) \in [e_{\HK}(I), e(I)]$. Note, if true, this would prove the regularity criteria. While we can't establish these facts, we give validation by showing the following.

\begin{mainthm}(Corollary~\ref{cor:constant})
For a local domain $(R,\fm,k)$ with residue field of characteristic $p > 0$, and $I$ an $\fm$-primary ideal, if $e(I) = e_{\HK}(I)$, then $e_s(I)$ is a constant function in $s$. 
\end{mainthm}

In general, it is not even clear from definitions that $e_s(R) \geq 1$! Despite this, we show a number of useful estimates which also establish the expected regularity criteria for all $s$ in dimension $2$. In search of more good behavior for intermediate multiplicities, we are inspired by general lower bounds similar to those arising in \cite{WY01}, \cite{WY05}. These papers heavily utilize an effective lower bound for Hilbert-Kunz multiplicity in terms of functions $\cH_s(d) :=  \sum_{i=0}^{ \lfloor s \rfloor} \frac{ (-1)^i}{d!} \binom{d}{i} (s-i)^d$. In the Cohen-Macaulay case we give an $s$-analogue of their lower bound. 

\begin{mainthm}(Theorem~\ref{thm:MainLowerBound})
Let $(R,\fm)$ be a $d$-dimensional Cohen-Macaulay local ring of characteristic $p>0$, let $I\subset R$ be an $\fm$-primary ideal, and let $J$ be a reduction of $I$ that is a parameter ideal.  For any $r\geq \mu(I/J^\ast)$ and $1\leq t \leq s$, 
\[e_s(I)\geq \left(\frac{\cH_t(d)-r\cH_{t-1}(d)}{\cH_s(d)}\right)e(I).\]
\end{mainthm}

The motivation for Watanabe and Yoshida was to formulate and give evidence of a general conjecture of deep interest. The Watanabe-Yoshida conjecture  predicts that unmixed singular local rings have Hilbert-Kunz multiplicities that are universally bounded below by those of the degree two Fermat hypersurfaces. In particular, working over a fixed base field $k$ of positive characteristic, set $$R_d := k[[x_0,\ldots,x_d]]/(x_0^2 + \ldots + x_d^2).$$ Specifically, the conjecture predicts that any singular unmixed complete local ring $R$ of dimension $d \geq 1$ has $e_{\HK}(R) \geq e_{\HK}(R_d)$ and equality forces $R$ to be analytically isomorphic to $R_d$. This conjecture holds in odd characteristics for all rings of dimension $d \leq 6$, for complete intersections, and for rings of small multiplicity, \cite{ES05}, \cite{AE13}. Clearly the corresponding Hilbert-Samuel statement is true and so one might hope for an $s$-analogue of this conjecture. 

\

{\bf Question:} Is $e_s(R) \geq e_s(R_d)$ for all singular complete unmixed rings of dimension $d$ and all $s >  0$? 

\

As in the Hilbert-Kunz case, a positive answer to this question gives an immediate target class of rings to study. Utilizing our general lower bound, we establish that this holds for Cohen-Macaulay local rings of dimension at most $3$. The Cohen-Macaulay hypothesis here only manifests through the general lower bound; Theorem~\ref{thm:MainLowerBound}. 

\begin{mainthm}(Theorem~\ref{sWY dim<4})
Let $(R,\fm)$ be an unmixed singular Cohen-Macaulay local ring of prime characteristic $p>0$ and dimension $d$ at most 3. For any $s>0$, $e_s(R)\geq e_s(R_d)$.
\end{mainthm}

We also note that many results in the literature nearly immediately give more evidence. For example, the question has a positive answer in the complete intersection case, by a mild extension of lemmas in \cite{ES05}, where the bulk of the heavy lifting there manifests in changes of coordinates and inductions which are independent of the underlying length calculations and so adds further evidence to a positive answer of the main question. %We also formulate an analogous question by examining closed forms for the function $s \mapsto e_s(R_d)$ for $d \leq 3$. These functions generalize to a particular function of $s$, notably independent of the characteristic, which serves as a lower bound among singular unmixed local ring of dimension $4$. We note however, that as this is independent of the characteristic, it cannot be the function $e_s(R_4)$. 

\

{\bf Acknowledgments:} We gratefully thank P. Mantero, M. Johnson, and J. McCullough for helpful discussions and Florian Enescu for reading a preliminary draft of the article. 

\section{A brief review of $s$-multiplicity}

We begin by recalling the $s$-multiplicity introduced in \cite{Tay}.  Throughout $s$ denotes a positive real number and $p$ denotes a positive prime integer. Fix a local ring $(R,\fm)$ of characteristic $p$ and two $\fm$-primary ideals $I$ and $J$. The limit of colengths $h_s(I,J) := \lim_{e\to\infty} \lambda( R / \sIJe{e} )/p^{ed}$ exists \cite[Thm. 2.1]{Tay}. It is related to both the Hilbert-Samuel and Hilbert-Kunz multiplicities of $I$ and $J$ respectively. In particular, for $s$ larger than the $F$-threshold $c_J(I)$, $h_s(I,J) = e_{\HK}(J)$ and for values of $s$ smaller than a  threshold of similar construction, $h_s(I,J) = \frac{s^d}{d!} e(I)$ \cite[Lem. 3.2]{Tay}. When $R$ is regular of dimension $d$, $\cH_s(d) := h_s(\fm,\fm) = \sum_{i=0}^{ \lfloor s \rfloor} \frac{ (-1)^i}{d!} \binom{d}{i} (s-i)^d$ offers a normalizing factor and one defines the $s$-multiplicity as follows.  

\begin{dff}
For a $d$-dimensional local ring $(R,\fm)$ of characteristic $p > 0$, the {\bf $s$-multiplicity} of $\fm$-primary ideals $I$ and $J$ is $e_s(I,J) := h_s(I,J)/\cH_s(d)$. 
\end{dff}

\noindent The $s$-multiplicity can be defined more generally for modules in the expected way and enjoys many of the usually expected properties. In the toric case, it has an interpretation as a Euclidean volume \cite[Lem. 5.3]{Tay} which will be  exploited in Example~\ref{xmp:R3}. We follow the usual conventions by writing $e_s(I) := e_s(I,I)$ and $e_s(R) := e_s(\fm)$.

It is unclear how the function $s \mapsto e_s(I,J)$ for fixed $I$ and $J$ behaves in general. It is always Lipschitz continuous \cite[Cor. 3.8(i)]{Tay} and takes the value $e(I)$ for small values of $s$ and $e_{\HK}(J)$ for large values of $s$. However, it is not known that this function is decreasing or even just bounded. While it is known that when $R$ is regular, $e_s(R) = 1$ for all $s$ \cite[Cor. 3.7]{Tay}, there is not yet a complete $s$-version of the Nagata-Watanabe-Yoshida regularity criteria for unmixed rings. 
%We utilized the term {\it formally unmixed local ring} for a local ring $R$ such that all associated primes of $\widehat{R}$ have the same dimension; the formal adjective being slightly non-standard but we feel expressive. 

%\newpage

\section{Estimations of $s$-multiplicity}

We start by calculating the $s$-multiplicity of a power of a maximal ideal in a regular local ring. 

\begin{lem}\label{lem:powermaxcomp}
Suppose $(B,\fn)$ is regular local with maximal ideal $\fn$. We have 
$$e_s(\fn^n) = \frac{ \sum_{a=0}^{n-1} \binom{a + d -1 }{ d-1} \cH_{sn-a}(d) } { \cH_s(d) }.$$
\end{lem}
\begin{proof}
We may complete and hence assume that $B=k[[x_1,\ldots,x_d]]$ and $\fn=(x_1,\ldots, x_d)$.  As $B$ is toric, \cite[Thm 5.4]{Tay} gives $e_s(\fn^n)=\vol(U)/\cH_s(d)$, where
\[U=\left\{(u_1,\ldots,u_d)\in \mathbf{R}_{\geq 0}^d\;|\;\sum_{i=1}^n u_i< sn, \sum_{i=1}^n \lfloor u_i\rfloor <n\right\}\]
By direct computation
\begin{align*}
\vol(U)  =\sum_{x\in\mathbf{Z}^d_{\geq 0}}\vol(U\cap(x+[0,1)^d))
=\sum_{a=0}^{n-1}\sum_{\stackrel{x\in\mathbf{Z}^d_{\geq 0},}{\|x\|_1=a}}\vol\left(\left\{y\in [0,1)^d\;|\;\|y\|_1<sn-a\right\}\right)=\sum_{a=0}^{n-1} \binom{a + d -1 }{ d-1} \cH_{sn-a}(d).
\end{align*}
\end{proof}

\begin{thm} Suppose $(B,\fn)$ is regular local with maximal ideal $\fn$.
For $s > 1$, $$\lim_{n \to \infty} \frac{e_s(\fn^n)}{n^d} = \frac{1}{d! \cH_s(d)}.$$ 
\end{thm}
\begin{proof}
If $n\geq \frac{d-1}{s-1}$, then for any $a\leq n-1$ we have that $sn-a\geq sn-n+1\geq d$, and so $\cH_{sn-a}(d)=1$.  Therefore, by Lemma \ref{lem:powermaxcomp},
\[\lim_{n \to \infty} \frac{e_s(\fn^n)}{n^d}=\lim_{n\to\infty}\frac{ \sum_{a=0}^{n-1} \binom{a + d -1 }{ d-1} } {n^d \cH_s(d) }=\lim_{n\to\infty}\frac{\binom{n + d -1 }{ d} } {n^d \cH_s(d) }=\frac{1}{d!\cH_s(d)}.\qedhere\]
\end{proof}

From this we show that $e_s(I)$ always has expected upper and lower bounds available for Hilbert-Kunz multiplicity and can establish the regularity criteria in dimension $2$.

\begin{lem}\label{lem:lowerestimate}
Let $(R,\fm,k)$ be a local ring of characteristic $p$ and $I$ an $\fm$-primary ideal 
\begin{enumerate}
\item $e_s(I) \geq \frac{ e(I)}{d!}$, 
\item If $1 < s < d$, then $e_s(I) > \frac{ e(I)}{d!}$,
\item When $R$ is singular, unmixed, and of dimension $2$, then $e_s(R) > 1$. 
\end{enumerate}
\end{lem}
\begin{proof}
The first two claims follow as the containment $I^{\lceil sp^e \rceil} + I^{[p^e]} \subset I^{p^e}$ holds for $s > 1$ and this forces $$e_s(I) \geq \lim_{e \to \infty} \frac{ \lambda(R/ (I^{\lceil sp^e \rceil} + I^{[p^e]})) }{ \cH_s(d)p^{ed} } \geq \lim_{e \to \infty} \frac{ \lambda( R/I^{p^e})}{ \cH_s(d) p^{ed}} = \frac{ e(I)}{ \cH_s(d) d!}.$$ Using this, the first claim holds as  $\cH_s(d) \leq 1$ for all $s$ and the second claim holds as $\mathcal{H}_s(d) < 1$ for $s < d$. The third claim follows from the first two using Nagata's theorem \cite[Thm. 40.6]{Nag62} for $s \leq 1$, claim (2) for $1 < s < d$, and Watanabe and Yoshida's theorem \cite[Thm. 1.5]{WY00} for $s \geq d$. 
\end{proof}

We next generalize a result of Watanabe and Yoshida. Notably, this holds without an numixed hypothesis.  %Notably, this holds {\it without} a formally unmixed hypothesis. 

\begin{thm}$($ c.f., \cite[Lem. 1.3]{WY01}$)$\label{thm:parameterflat}
Fix $(R,\fm)$ a local $d$-dimensional ring that is either a domain or Cohen-Macaulay. For $J$ a parameter ideal and $n\in \NN$, $$e_s(J^n) =  \frac{ \sum_{a=0}^{n-1} \binom{a + d -1 }{ d-1} \cH_{sn-a}(d) } { \cH_s(d) } e(J).$$  If $R/\fm$ is infinite and $I$ is an $\fm$-primary ideal, then for any $s>0$ and $n\in \NN$, we have \begin{equation}\label{eq1}\frac{e(I^n)}{d!} \leq e_s(I^n) \leq \frac{ \sum_{a=0}^{n-1} \binom{a + d -1 }{ d-1} \cH_{sn-a}(d)}{n^d \cH_s(d)} e(I^n).\end{equation} In particular, $\frac{e(I)}{d!} \leq e_s(I) \leq e(I)$ for all $s$. 
\end{thm}
\begin{proof}
Note that \cite[Lem. 2.3]{WY05}, in our notation states for a parameter ideal $J$, $e_s(J) = e(J)$. Let $a_1,\ldots, a_d$ be a system of parameters generating $J$, and let $S=k[[a_1,\ldots, a_d]]\subseteq R$.  If $R$ is Cohen-Macaulay, then by \cite[Prop 3.12]{Tay} and the calculation in the proof of Lemma~\ref{lem:powermaxcomp},
\[e_s^R(J^n)= e_s^S((a_1,\ldots ,a_d)^n)\lambda(R/J)= \frac{ \sum_{a=0}^{n-1} \binom{a + d -1 }{ d-1} \cH_{sn-a}(d) } {  \cH_s(d) }e(J).\]
If $R$ is a domain then we can use \cite[Prop 3.11]{Tay} and \cite[Theorem 11.2.7]{HS06} to conclude that
\[e_s^R(J^n)=e_s^S((a_1,\ldots, a_d)^n)\mathrm{rank}_S(R)=\frac{ \sum_{a=0}^{n-1} \binom{a + d -1 }{ d-1} \cH_{sn-a}(d) } {  \cH_s(d) }e(J).\]
Equation~\eqref{eq1} follows from Lemma~\ref{lem:lowerestimate} and the first claim. %The first claim follows the same proof of \cite[Lem 1.3]{WY01} replacing \cite[Thm. 2.7]{WY00} by the $s$-version \cite[Prop. 3.10]{Tay} and the computation of $e_s(\fn^n)$ for $\fn$ a maximal ideal in a $d$-dimensional regular local ring by Lemma~\ref{lem:powermaxcomp}. 
In particular, taking $J \subset I$ a minimal reduction,
$$e_s(I^n) \leq e_s(J^n)  =\frac{ \sum_{a=0}^{n-1} \binom{a + d -1 }{ d-1} \cH_{sn-a}(d) } {  \cH_s(d) } e(J) = \frac{ \sum_{a=0}^{n-1} \binom{a + d -1 }{ d-1} \cH_{sn-a}(d) } {  \cH_s(d) }e(I) = \frac{ \sum_{a=0}^{n-1} \binom{a + d -1 }{ d-1} \cH_{sn-a}(d) } {  n^d\cH_s(d) } e(I^n).$$

\end{proof}

Immediately, we obtain the first main result of the paper.

\begin{cor}$($ c.f., \cite[Thm. 1.8]{WY01}$)$\label{cor:constant}
Suppose $(R,\fm)$ is a complete local domain of dimension at least $2$. If $I$ is an $\fm$-primary ideal, then $e_s(I)$ is constant if and only if $e_{\HK}(I) = e(I)$. 
\end{cor}
\begin{proof}
The following proof is inspired by \cite[Thm 1.8 (3)]{WY01}.
 Assuming $e_{\HK}(I) = e(I)$, we have a parameter ideal $J$ such that $J \subset I \subset J^*$, and therefore $e_s(J)\geq e_s(I)\geq e_s(J^*)=e_s(J)$ and we have equality throughout. Therefore, by Theorem~\ref{thm:parameterflat},
$e_s(I)=e_s(J)=e(J)=e(I)$ for all $s$.  The converse direction is obvious as $e(I)=e_1(I)$ and $e_{\HK}(I)=e_d(I)$.
%For each $n$, we have the usual inclusions $J^n \subset I^n \subset (J^*)^n \subset (J^n)^*$ which gives inequalities for fixed $s$, $$e_s(J^n) \geq e_s(I^n) \geq e_s( (J^n)^*) = e_s(J^n)$$ where the last equality follows as quite generally $e_s(P,Q) \geq e_s( P^*, Q^*) \geq e_s( \overline{P}, Q^*) = e_s( P, Q)$ by \cite[Cor. 3.8(v, vi) ]{Tay}. The rest follows by Theorem~\ref{thm:parameterflat}. Clearly $e_s(I)$ being constant implies $e_{\HK}(I) = e(I)$. Conversely, we have for each $s$, $e_s(I) = e(I)$ and so $e_s(I)$ is constant. 
\end{proof}

The following lemma is an $s$-analogue of a heavily utilized result for Hilbert-Kunz multiplicity, see  \cite[Cor. 2.2(2)]{HY02}, \cite[Lem. 4.2]{WY00}, and \cite[Lem. 2.2 (2)]{BE04}.

\begin{lem} \label{I-J inequality} If $(R,\fm)$ is a local ring of prime characteristic $p>0$ and dimension $d$ and $I$ and $J$ are $\fm$-primary ideals of $R$ such that $I\subseteq J\subseteq \overline{I}$, then
$e_s(I)\leq e_s(J)+\lambda(J/I)e_s(I,\fm)$.
\end{lem}

\begin{proof} We proceed by induction on $\lambda(J/I)$.  Suppose that $\lambda(J/I)=1$, and let $x\in R$ generate $J/I$.  For any $q=p^e$, we have that
\begin{align*}
h_s(I)-h_s(J)  = h_s(I)-h_s(I,J) &= \lim_{q\to\infty}\frac{1}{q^d}\left(\lambda\left(R/I^{\lceil sq\rceil}+I^{[q]}\right)-\lambda\left(R/I^{\lceil sq\rceil}+J^{[q]}\right)\right)\\&= \lim_{q\to\infty}\frac{1}{q^d}\lambda\left(\frac{I^{\lceil sq\rceil}+J^{[q]}}{I^{\lceil sq\rceil}+I^{[q]}}\right)
= \lim_{q\to\infty}\frac{1}{q^d}\lambda\left(\frac{I^{\lceil sq\rceil}+I^{[q]}+(x^q)}{I^{\lceil sq\rceil}+I^{[q]}}\right).
\end{align*}
Since $\lambda(J/I)=1$, $\fm x\subseteq I$, and so $\fm^{[q]} x^q\subseteq I^{[q]}$.  Therefore $I^{\lceil sq\rceil}+\fm^{[q]}$ annihilates the principal module $\frac{I^{\lceil sq\rceil}+I^{[q]}+(x^q)}{I^{\lceil sq\rceil}+I^{[q]}}$, and hence $\lambda\left(\frac{I^{\lceil sq\rceil}+I^{[q]}+(x^q)}{I^{\lceil sq\rceil}+I^{[q]}}\right)\leq \lambda(R/I^{\lceil sq\rceil}+\fm^{[q]})$.  Therefore,
\[h_s(I)-h_s(J)= \lim_{q\to\infty}\frac{1}{q^d}\lambda\left(\frac{I^{\lceil sq\rceil}+I^{[q]}+(x^q)}{I^{\lceil sq\rceil}+I^{[q]}}\right)\leq \lim_{q\to\infty}\frac{1}{q^d}\lambda(R/I^{\lceil sq\rceil}+\fm^{[q]})=h_s(I,\fm).\]
Now suppose that $\lambda(J/I)\geq 2$.  Let $K$ be an ideal of $R$ such that $I\subsetneq K\subsetneq J$.  By induction, we have 
\[e_s(I)\leq e_s(K)+\lambda(K/I)e_s(I,\fm)\leq e_s(J)+\lambda(J/K)e_s(K,\fm)+\lambda(K/I)e_s(I,\fm)=e_s(J)+\lambda(J/I)e_s(I,\fm).\qedhere\]
\end{proof}

\begin{cor}\label{cor:AtLeast1} If $(R,\fm)$ is a Cohen-Macaulay local ring of prime characteristic $p>0$, then $e_s(R)\geq 1$.
\end{cor}

\begin{proof}
We may assume that $R/\fm$ is infinite.  Let $I$ be a minimal reduction of $\fm$.  We have, by Lemma \ref{I-J inequality}, that $e_s(I)\leq e_s(\fm)+\lambda (\fm/I)e_s(I,\fm)=e_s(\fm)\lambda(R/I)$ where the last equalty holds by \cite[Prop. 2.6(v)]{Tay} and additivity of length on short exact sequences. Since $I$ is a parameter ideal and $R$ is Cohen-Macaulay, we have $e_s(I)=e(I)=\lambda(R/I)$, which finishes the proof.
\end{proof}

\subsection{General Lower Bounds}

We conclude this section with an important lower bound for $e_s(I)$ for general $s$ and $\fm$-primary $I$ in Cohen-Macaulay rings which matches in spirit the lower bounds explored by \cite[Thm. 2.2]{WY05}. In some ways, working with $s$-multiplicity simplifies their arguments, which were suggestive all along of this type of multiplicity. 

\begin{thm}\label{thm:MainLowerBound}
Let $(R,\fm)$ be a $d$-dimensional Cohen-Macaulay local ring of characteristic $p>0$, let $I\subset R$ be an $\fm$-primary ideal, and let $J$ be a reduction of $I$ that is a parameter ideal.  For any $r\geq \mu(I/J^\ast)$ and $1\leq t \leq s$, 
\[e_s(I)\geq \left(\frac{\cH_t(d)-r\cH_{t-1}(d)}{\cH_s(d)}\right)e(I).\]
\end{thm}

\begin{proof} Throughout $q$ denotes a power of $p$ and references to limits with $q \to \infty$ assume $q$ moves through prime powers. For any fixed $q$, we have that
\begin{align*}
\lambda\left(\frac{I^{\lceil sq\rceil} +I^{[q]}}{J^{\lceil sq\rceil}+J^{[q]}}\right)
&\leq \lambda\left(\frac{I^{\lceil tq\rceil} +I^{[q]}}{J^{\lceil sq\rceil}+J^{[q]}}\right)\\
&=\lambda\left(\frac{I^{\lceil tq\rceil} +I^{[q]}}{I^{\lceil tq\rceil}+(J^\ast)^{[q]}}\right) 
+ \lambda\left(\frac{I^{\lceil tq\rceil} +(J^\ast)^{[q]}}{J^{\lceil tq\rceil}+(J^\ast)^{[q]}}\right)
+ \lambda\left(\frac{J^{\lceil tq\rceil}+(J^\ast)^{[q]}}{J^{\lceil tq\rceil} +J^{[q]}}\right)
+\lambda\left(\frac{J^{\lceil tq\rceil} +J^{[q]}}{J^{\lceil sq\rceil}+J^{[q]}}\right).
\end{align*}
If $x$ is a generator of $I/J^\ast$, then $\frac{I^{\lceil tq\rceil} +(J^\ast)^{[q]}+(x^{q})}{I^{\lceil tq\rceil}+(J^\ast)^{[q]}}$ is annihilated by $I^{\lceil (t-1)q\rceil}+(J^\ast)^{[q]}$, and so,
\begin{align*}
\lim_{q\to\infty}\frac{1}{q^d}\lambda\left(\frac{I^{\lceil tq\rceil} +I^{[q]}}{I^{\lceil tq\rceil}+(J^\ast)^{[q]}}\right)
&\leq \lim_{q\to\infty}\frac{1}{q^d}r\cdot \lambda\left(R/(I^{\lceil (t-1)q\rceil}+(J^\ast)^{[q]})\right)\\&=r\cdot h_{t-1}(I,J^\ast)=rh_{t-1}(J)=r\cH_{t-1}(d)\lambda(R/J)=r\cH_{t-1}(d)e(I).\end{align*}
Since $J$ is a minimal reduction of $I$, $\lambda\left(\frac{I^{\lceil tq\rceil} +(J^\ast)^{[q]}}{J^{\lceil tq\rceil}+(J^\ast)^{[q]}}\right)=o(q^d)$.
Furthermore, we have that $\lambda\left(\frac{J^{\lceil tq\rceil}+(J^\ast)^{[q]}}{J^{\lceil tq\rceil} +J^{[q]}}\right)= o(q^d)$.
Finally, since $J$ is a parameter ideal and $R$ is Cohen-Macaulay,
\[\lim_{q\to\infty}\frac{1}{q^d}\lambda\left(\frac{J^{\lceil tq\rceil} +J^{[q]}}{J^{\lceil sq\rceil}+J^{[q]}}\right) = h_s(J)-h_t(J)=(\cH_s(d)-\cH_t(d))\lambda(R/J)=(\cH_s(d)-\cH_t(d))e(I).\]
All the above together implies
\begin{align*}
h_s(J)-h_s(I)
=\lim_{q\to \infty}\frac{1}{q^d}\lambda\left(\frac{I^{\lceil sq\rceil}+I^{[q]}}{J^{\lceil sq\rceil}+J^{[q]}}\right)
\leq r\cH_{t-1}(d)e(I)+(\cH_s(d)-\cH_t(d))e(I).
\end{align*}
Since $e_s(J)=\lambda (R/J)=e(I)$, 
\begin{align*}e_s(I)=\frac{h_s(I)}{\cH_s(d)}
&\geq \frac{1}{\cH_s(d)}\left(h_s(J)- \left(r\cH_{t-1}(d)+\cH_s(d)-\cH_t(d)\right)e(I)\right)\\
&=e_s(J)-\left(r\frac{\cH_{t-1}(d)}{\cH_s(d)}+1-\frac{\cH_t(d)}{\cH_s(d)}\right)e(I)
= \left(\frac{\cH_t(d)-r\cH_{t-1}(d)}{\cH_s(d)}\right)e(I).\qedhere
\end{align*}
\end{proof}

\begin{rmk} The following observation is a key part of the arguments in \cite{WY05}. If $R$ is Cohen-Macaulay and $J$ is a reduction of $I$ that is a parameter ideal, then $\mu(I/J^*)\leq e(I)-1$, which can be easily seen:
$$\mu(I/J^*) \leq \lambda(I/J^*) = \lambda(R/J^*) - \lambda(R/I) \leq  \lambda(R/J) - 1 = e(J)-1=e(I) -1.$$  If $R$ is not $F$-rational, then we can do even better:
$$\mu(I/J^*) \leq \lambda(I/J^*) = \lambda(R/J) -\lambda(J^*/J)- \lambda(R/I) \leq  \lambda(R/J) - 2 = e(J)-2=e(I) -2.$$
This means one maximizes the lower bound in the case with $r = e(I)-1$ in the $F$-rational case and $r=e(I)-2$ in the non-$F$-rational case. In the next section, we apply this with $I = \fm$. Note in particular, this result gives most information for rings of multiplicity $e \geq 2$ as when $e = 1$ we recover only a lower bound of $1$. %Although this may be viewed as another proof of Corollary~\ref{cor:AtLeast1} in the unmixed case. 

\end{rmk}

\section{An $s$-analogue of a conjecture by Watanabe and Yoshida}

Our primary application of the results of the last section, in particular Theorem~\ref{thm:MainLowerBound}, is towards minimal values for singular rings. Throughout this section, for $k$ a field, set $R_d = k[[x_0,\ldots,x_d]]/\sum x_i^2$ for $d \geq 1$. The Watanabe-Yoshida conjecture predicts that these rings enjoy the absolute minimal Hilbert-Kunz multiplicity among unmixed singular rings of that fixed dimension. This is known  for all rings of dimension at most $4$ and for $p > 2$ for rings of dimension at most $6$ by \cite{WY05, AE13}, for complete intersections by \cite{ES05}, and for rings in any dimension of multiplicity $5$ by \cite{AE13}. Moreover, rings achieving the minimal value are conjectured to be analytically isomorphic to $R_d$. Inspired, we ask the following question. 

\begin{que}\label{que:BigQuestion} Fix $R$ an unmixed local ring of dimension $d \geq 1$.  
\begin{enumerate}
\item Is $e_s(R) \geq e_s(R_d)$ for all $s$? 
\item Does equality for any fixed $s$ force $R$ to be analytically isomorphic to $R_d$?
\end{enumerate}
\end{que}

We exploit the lower bound in Theorem~\ref{thm:MainLowerBound}, analogous to the work of Watanabe-Yoshida and Aberbach-Enescu, to give evidence of a positive answer to this question in the Cohen-Macaulay setting and dimension at most $3$ as well as complete intersections in all dimensions. Before doing so, we recall by \cite[Example 5.6]{Tay} the $s$-multiplicity of the $n$-th Veronese of $k[[x,y]]$. For $n=2$ we have 
\[e_s(R_2)=\begin{cases} 2-\frac{2\cH_{s-1}(2)}{\cH_s(2)} & \text{if } s\leq \frac{3}{2} \\ \frac{3}{2\cH_s(2)} & \text{if } s\geq \frac{3}{2}.\end{cases}\] This follows from a direct volume computation. Likewise, we have need an explicit computation of $e_s(R_3)$. To streamline the exposition, we have included a more detailed version of this computation in the Appendix. 

\begin{xmp}\label{xmp:R3} The $s$-multiplicity of $R_3$ is given by 
\[e_s(R_3)=\begin{cases} 2-\frac{2\cH_{s-1}(3)}{\cH_s(3)} & \text{if } s\leq 2 \\ \frac{4}{3\cH_s(3)} & \text{if } s\geq 2.\end{cases}\] This computation lies in recognizing $R_3 \cong k[[x,y,z,w]]/(xy-zw)$. This allows one to express $R_3$ as an affine semigroup ring, defined by a cone $\sigma$. The resulting length computation then follows by computing the volume of the set \[U=\left\{v\in \sigma^\vee \;|\;v\notin s\cdot \Hull \fm \cup (\Exp \fm+\sigma^\vee)\right\}\] where we recall that $\Exp \fm$ is the set of semigroup elements arising in $\fm$ and $\Hull \fm$ is the convex hull of $\Exp \fm$. Routine toric computations, explained in the appendix, expresses this volume for $s \leq 2$ as an integral 
\begin{align*}
h_s(R_3)=\vol(U)&=2\int_0^1\cH_{s-z}(2)-z^2\cH_{1-(2-s)/z}(2)dz=2\cH_s(3)-2\int_{2-s}^1\frac{1}{2}z^2\left(1-\frac{2-s}{z}\right)^2dz\\
&=2\cH_s(3)-\int_{2-s}^1(z-(2-s))^2dz=2\cH_s(3)-\frac{(s-1)^3}{3}=2\cH_s(3)-2\cH_{s-1}(3). 
\end{align*} Furthermore, the $F$-threshold of $\fm$ with respect to $\fm$ is at most $2$. Hence, by \cite[Lem. 3.2]{Tay}, $h_s(R_3)= e_{\HK}(R_3)$ in that range. 
\end{xmp}

Armed with these computations, we provide our main supporting evidence towards Question~\ref{que:BigQuestion} (1). 

\begin{lem} \label{small s inequality} If $R$ is a singular Cohen-Macaulay ring of dimension $d$ and multiplicity $e$ and $s\leq \min\left\{2,\frac{\sqrt[d]{d+e+1}}{\sqrt[d]{d+e+1}-1}\right\}$, then $$e_s(R)\geq 2-\frac{2\cH_{s-1}(d)}{\cH_{s}(d)}.$$
\end{lem}

\begin{proof} In this case, we have that $0\leq \frac{s^d}{d!}-\frac{(d+e+1)(s-1)^d}{d!}=\cH_s(d)-(e+1)\cH_{s-1}(d)$,
and consequently it is straightforward to check 
$$e_s(R)\geq \frac{\left(\cH_{s}(d)-(e-1)\cH_{s-1}(d)\right)e}{\cH_s(d)}\geq \frac{2\cH_s(d)-2\cH_{s-1}(d)}{\cH_s(d)}=2-\frac{2\cH_{s-1}(d)}{\cH_{s}(d)}.\qedhere$$
\end{proof}

\begin{thm} \label{sWY dim<4} Let $(R,\fm)$ be an unmixed singular Cohen-Macaulay local ring of prime characteristic $p>0$ and dimension $d$ at most 3. For any $s>0$, $e_s(R)\geq e_s(R_d)$.
\end{thm}

\begin{proof} The case where $d = 1$ follows as $e_s(R)=e(R)\geq 2=e(R_d)=e_s(R_d)$ by Lemma~\ref{lem:lowerestimate}. Suppose $d=2$ and recall  \[e_s(R_2)=\begin{cases} 2-\frac{2\cH_{s-1}(2)}{\cH_s(2)} & \text{if } s\leq \frac{3}{2} \\ \frac{3}{2\cH_s(2)} & \text{if } s\geq \frac{3}{2}.\end{cases}\]
Set $e=e(R)$. If $e\geq 4$, then $e_s(R)\geq \frac{e}{2}\geq 2 =e(R_d) \geq e_s(R_d)$ by Lemma~\ref{lem:lowerestimate} and Theorem~\ref{thm:parameterflat}.  If $e=2$ or $e=3$, then $\frac{\sqrt{e+3}}{\sqrt{e+3}-1}\geq \frac{3}{2}$, and so if $s\leq \frac{3}{2}$ then $e_s(R)\geq 2-\frac{2\cH_{s-1}(2)}{\cH_s(2)}=e_s(R_2)$ by Lemma~\ref{small s inequality}.
%The rest is a consequence of Theorem~\ref{thm:MainLowerBound} with $I = \fm$ and $t = s$. Recall that we may take $r = e-1$ %by remark
%and so $$e_s(R)\geq \left(\frac{\cH_s(2)-(e-1)\cH_{s-1}(2)}{\cH_s(2)}\right)e.$$ It is straightforward to check that, {\color{brown} for $s\leq 2$,} $\left(\frac{\cH_s(2)-(e-1)\cH_{s-1}(2)}{\cH_s(2)}\right)e  \geq 2-\frac{2\cH_{s-1}(2)}{\cH_s(2)}$ happens if and only if $s\leq \frac{\sqrt{e+3}}{\sqrt{e+3}-1}$. 
% \begin{align*}
%    & \left(\frac{\cH_s(2)-(e-1)\cH_{s-1}(2)}{\cH_s(2)}\right)e  \geq 2-\frac{2\cH_{s-1}(2)}{\cH_s(2)} \\ 
%     \Leftrightarrow &(e-2)\cH_s(2) \geq (e(e-1)-2)\cH_{s-1}(2) \\ 
%     \Leftrightarrow& \cH_s(2)\geq (e+1)\cH_{s-1}(2)\\
%     \Leftrightarrow &\frac{s^2}{2}-(s-1)^2\geq \frac{e+1}{2}(s-1)^2\\
%     \Leftrightarrow &s\leq \frac{\sqrt{e+3}}{\sqrt{e+3}-1}
% \end{align*}
%Now suppose $s\leq \frac{3}{2}$. If $e=2$ or $e=3$ then $\frac{\sqrt{e+3}}{\sqrt{e+3}-1}>\frac{3}{2}$ and so $e_s(R)\geq e_s(R_d)$ for all such $s$. 
If $e=2$ and $s\geq \frac{3}{2}$, then $e_s(R)\geq \left(\frac{\cH_{\frac{3}{2}}(2)-\cH_{\frac{1}{2}}(2)}{\cH_s(2)}\right)2=\frac{3}{2\cH_s(2)}$, and if $e=3$ and $s\geq \frac{3}{2}$, then $e_s(R)\geq \left(\frac{\cH_{1}(2)-2\cH_{0}(2)}{\cH_s(2)}\right)3=\frac{3}{2\cH_s(2)}$ by Theorem~\ref{thm:MainLowerBound}.  This finishes the case $d=2$.

Now suppose that $d=3$. Recall by Example~\ref{xmp:R3}, we have
\[e_s(R_3)=\begin{cases} 2-\frac{2\cH_{s-1}(3)}{\cH_s(3)} & \text{if } s\leq 2 \\ \frac{4}{3\cH_s(2)} & \text{if } s\geq 2.\end{cases}\] 

We first handle cases of large multiplicity. If $e\geq 12$, then $e_s(R)\geq \frac{e}{3!}\geq 2=e(R_d)\geq e_s(R_d)$ by Lemma~\ref{lem:lowerestimate} and Theorem~\ref{thm:parameterflat} so it suffices to assume $e < 12$. Again applying Theorem~\ref{thm:MainLowerBound} we have for all $s$, $$e_s(R) \geq \left(\frac{\cH_s(3)-(e-1)\cH_{s-1}(3)}{\cH_s(3)}\right)e.$$

% \begin{align*}
%    & \left(\frac{\cH_s(3)-(e-1)\cH_{s-1}(3)}{\cH_s(3)}\right)e  \geq 2-\frac{2\cH_{s-1}(3)}{\cH_s(3)} \\ 
%     \Leftrightarrow &(e-2)\cH_s(3) \geq (e(e-1)-2)\cH_{s-1}(3) \\ 
%     \Leftrightarrow& \cH_s(3)\geq (e+1)\cH_{s-1}(3)\\
%     \Leftrightarrow &\frac{s^3}{6}-\frac{(s-1)^2}{2}\geq \frac{e+1}{6}(s-1)^3\\
%     \Leftrightarrow &s\leq \frac{\sqrt[3]{e+4}}{\sqrt[3]{e+4}-1}.
% \end{align*}

%For all $e \geq 1$ and $s \leq \frac{\sqrt{e+2}}{\sqrt{e+2}-1} \leq \frac{\sqrt[3]{e+4}}{\sqrt[3]{e+4}-1}$, we have $e_s(R)\geq e_s(R_d)$.

 %Let $t=\frac{\sqrt{e+2}}{\sqrt{e+2}-1}\leq 2$.  Suppose $s\leq t$. Since $\frac{\sqrt{e+2}}{\sqrt{e+2}-1}\leq \frac{\sqrt[3]{e+4}}{\sqrt[3]{e+4}-1}$ for all $e\geq 1$, we have that $e_s(R)\geq e_s(R_d)$ for $s\leq t$.

%It is again straightforward to check that $\left(\frac{\cH_s(3)-(e-1)\cH_{s-1}(3)}{\cH_s(3)}\right)e  \geq 2-\frac{2\cH_{s-1}(3)}{\cH_s(3)} $ if and only if $s\leq \frac{\sqrt[3]{e+4}}{\sqrt[3]{e+4}-1}.$ 
For all $e \geq 2$, if $s \leq \frac{\sqrt{e+2}}{\sqrt{e+2}-1}$, then $s \leq \frac{\sqrt[3]{e+4}}{\sqrt[3]{e+4}-1}$, and so $e_s(R)\geq e_s(R_d)$ by Lemma~\ref{small s inequality}. Suppose now that $ \frac{\sqrt{e+2}}{\sqrt{e+2}-1} < s$. We again apply Theorem~\ref{thm:parameterflat} and analyize the lower bound. For $t \leq s$ we have 
\[\left(\cH_t(3)-(e-1)\cH_{t-1}(3)\right)e
=\left(\frac{t^3}{6}-\frac{1}{2}(t-1)^2-\frac{e-1}{6}(t-1)^3\right)e=\frac{e}{6}\left(t^3-(e+2)(t-1)^3\right)\]
and so applying this with $t =  \frac{\sqrt{e+2}}{\sqrt{e+2}-1}$ we have 
\begin{eqnarray*}
\left(\cH_t(3)-(e-1)\cH_{t-1}(3)\right)e & =& \frac{e}{6}\left(\left(\frac{\sqrt{e+2}}{\sqrt{e+2}-1}\right)^3-(e+2)\left(\frac{1}{\sqrt{e+2}-1}\right)^3\right)\\
&=&\frac{e(e+2)}{6(\sqrt{e+2}-1)^3}\left(\sqrt{e+2}-1\right)=\frac{e(e+2)}{6(\sqrt{e+2}-1)^2}.\end{eqnarray*} As $t = \frac{\sqrt{e+2}}{\sqrt{e+2}-1} \leq 2$ to show $e_s(R) \geq e_s(R_3)$ for all $t < s$ we must address cases comparing $s$ versus $2$. However, it is straightforward to check that 
\begin{align*}
   \left(\frac{\cH_t(3)-(e-1)\cH_{t-1}(3)}{\cH_s(3)}\right)e  \geq 2-\frac{2\cH_{s-1}(3)}{\cH_s(3)}
    & \text{ if and only if }  \frac{e(e+2)}{6(\sqrt{e+2}-1)^2} \geq 2\cH_s(3)-2\cH_{s-1}(3)\quad \text{and} \\ 
       \left(\frac{\cH_t(3)-(e-1)\cH_{t-1}(3)}{\cH_s(3)}\right)e  \geq \frac{4}{3\cH_s(3)}
    & \text{ if and only if }  \frac{e(e+2)}{6(\sqrt{e+2}-1)^2} \geq \frac{4}{3}.
\end{align*} When $s\leq 2$, $2\cH_s(3)-2\cH_{s-1}(3)=\frac{s^3}{3}-\frac{4}{3}(s-1)^3\leq \frac{4}{3}$.  Thus to show that $e_s(R)\geq e_s(R_3)$ for $t < s$ it suffices to show for all $e \geq 2$ that $\frac{e(e+2)}{6(\sqrt{e+2}-1)^2}\geq \frac{4}{3}$. %The left hand side of \eqref{eq:neededIE} is increasing as a function of $e$ and thus for $e \geq 2$ by direct calculation
However, this is a simple computation:
\begin{align*}
    \frac{e(e+2)}{(\sqrt{e+2}-1)^2} & = \frac{(e+1)^2-1}{(\sqrt{e+2}-1)^2}
    =(\sqrt{e+2}+1)^2-\frac{1}{(\sqrt{e+2}-1)^2}
    \geq (\sqrt{e+2}+1)^2-1
    \geq (\sqrt{4}+1)^2-1
    =8.
\end{align*}
Therefore, $\frac{e(e+2)}{6(\sqrt{e+2}-1)^2}\geq \frac{4}{3}$ as desired.
%\begin{tabular}{|c|c|}
%\hline
%    $e$ & $\frac{e(e+2)}{6(\sqrt{e+2}-1)^2}$ \\ 
%    \hline
%    2 & $\frac{4}{3}$\\
 %   3 & $\frac{5}{2(\sqrt{5}-1)^2}\approx 1.63$\\
 %   4 & $\frac{4}{(\sqrt{6}-1)^2}\approx 1.90$\\
 %   5 & $\frac{35}{6(\sqrt{7}-1)^2}\approx 2.15$\\
 %   6 & $\frac{8}{(\sqrt{8}-1)^2}\approx 2.39$\\
 %   7 & $\frac{21}{8}\approx 2.62$\\
 %   8 & $\frac{40}{3(\sqrt{10}-1)^2}\approx 2.85$\\
 %   9 & $\frac{33}{2(\sqrt{11}-1)^2}\approx 3.07$\\
 %   10 & $\frac{20}{(\sqrt{12}-1)^2}\approx 3.29$\\
 %   11 & $\frac{143}{6(\sqrt{13}-1)^2}\approx 3.51$\\
 %   \hline
%\end{tabular}
\end{proof}

We can say even more in the dimension 2 case.

\begin{thm} Let $R$ be a 2-dimensional unmixed Cohen-Macaulay local ring of prime characteristic that is not $F$-rational. If $s>1$, then $e_s(R) > e_s(V_e)$, where $e=e(R)$ and $V_e$ is the $e$-th Veronese of $k[[x,y]]$.
\end{thm}

\begin{proof} We have, by \cite[Ex. 5.6]{Tay}, that 
\[e_s(V_e)=\begin{cases} \frac{e\cH_s(2)-(e^2-e)\cH_{s-1}(2)}{\cH_s(2)} &  0<s\leq \frac{e+1}{e}\\ \frac{e+1}{2\cH_s(2)} & s\geq \frac{e+1}{e}.\end{cases}\]
Suppose $1<s\leq \frac{e+1}{e}$.  Applying Theorem~\ref{thm:MainLowerBound} with $r=e-2$ and $t=s$ we have that
\[e_s(R)\geq \left(\frac{\cH_s(2)-(e-2)\cH_{s-1}(2)}{\cH_s(2)}\right)e>\left(\frac{\cH_s(2)-(e-1)\cH_{s-1}(2)}{\cH_s(2)}\right)e=e_s(V_e).\]
Now suppose that $s>\frac{e+1}{e}$.  We have that
\[\cH_{(e+1)/e}(2)-(e-2)\cH_{1/e}(2)=\frac{(e+1)^2}{2e^2}-\frac{1}{e^2}-\frac{e-2}{2e^2}=\frac{(e+1)^2-e}{2e^2}=\frac{e^2+e+1}{2e^2},\]
and therefore
\[e_s(R)\geq \left(\frac{\cH_{(e+1)/e}(2)-(e-2)\cH_{1/e}(2)}{\cH_s(2)}\right)e=\frac{e^2+e+1}{2e\cH_s(2)}>\frac{e+1}{2\cH_s(2)}.\qedhere\]
\end{proof}
\subsection{Complete intersections}

Much of the work of \cite{AE13} towards the Watanabe-Yoshida conjecture is based on computations first done in \cite{ES05} on complete intersections. Much of the work there carries over to the intermediate $s$-multiplicities nearly verbatim. The key point is to use various complicated changes of variables to directly prove the Watanabe-Yoshida conjecture. For example, when $R = S/f$ is a hypersurface where $S = k[[x_0,\ldots,x_d]]$, \cite[Thm. 3.2]{ES05} shows when $p > 2$ one has $e_{\HK}(R) \geq e_{\HK}(R_d)$ which follows by repeated application of changes of variables, hence the restriction on characteristic, to transform the defining hypersurface into a sum of squares. At each stage, a careful choice of regular element $g$ is made, along with a general field element $\alpha \in k$, and the inequality follows by \cite[Thm. 2.6]{ES05} which guarantees $e_{\HK}( S/(f+\alpha g)) \leq e_{\HK}(R)$. We show next that the proof of \cite[Thm. 2.6]{ES05} goes through for $s$-multiplicity and consequently the $s$-version of \cite[Thm. 3.2]{ES05} immediately follows. 

\begin{thm}\label{thm:shypersurface} Suppose $k$ is algebraically closed of characteristic $p > 2$. Set $S = k[[x_0,\ldots,x_n]]$. Fix $f, g \in (x_0,\ldots,x_n)$ so that $g \notin (f)$. There is a dense set $\Lambda \subset k$, in its finite complement topology, such that $$e_s(S/(f + \alpha g) ) \leq e_s(S/f)$$ for all $\alpha \in \Lambda$ and all $s$. 
\end{thm}
\begin{proof} This is an obvious translation of the proof of \cite[Thm. 3.2]{ES05}, but we give the details. Set $R = S[t]/(f+tg)$, $\fm=(x_0,\ldots,x_n)$. For $\alpha \in k$, set $\fm_\alpha = (x_0,\ldots,x_n,t-\alpha)$, so $\fm_0 = \fm$. For each fixed $q$,  $A = R/ (\fm^{ \lceil sq \rceil} + \fm^{[q]})$.  For $\alpha \in k$, $A/(t-\alpha) \cong R/( \fm_\alpha^{\lceil sq \rceil} + \fm_\alpha^{[q]} + (t-\alpha) )$ is Artinian and Nakayama's lemma gives its $k$-dimension is the minimal number of generators of $A$ localized away from the prime ideal $(t-\alpha)$, i.e., $A_{(t-\alpha)}$. Starting with a minimal set of generators $A_{(t)}$, i.e., for which $\alpha = 0$, we may find an open set $\Lambda_q$ in $\AA_k^1$ over which these generators also generate $A_{(t-\alpha)}$, possibly non-minimally, for all $(t-\alpha) \in \Lambda_q$. We may freely identify $\AA_k^1$ with $k$ by identifying $(t-\alpha)$ and $\alpha$, for which the Zariski topology on $\AA_k^1$ is the finite complement topology on $k$. Now setting $\Lambda = \bigcap \Lambda_q$, one has for all $\alpha \in \Lambda$ and all $q$, $$\lambda( R/ ( \fm_\alpha^{\lceil sq \rceil} + \fm_\alpha^{[q]} + (t-\alpha) )) \leq \lambda(R/ (\fm^{\lceil sq \rceil} + \fm^{[q]} + (t) )).$$ Taking limits as $q$ goes to infinity and normalizing give the result. 
\end{proof}

\begin{cor}
For any $d$-dimensional singular hypersurface $R \cong k[[x_0,\ldots,x_n]]/f$ over an uncountable algebraically closed field $k$ of characteristic $p > 2$, $e_s( R) \geq e_s(R_d)$ for all $s > 0$.
\end{cor}
\begin{proof}
This just follows the proof \cite[Thm. 3.2]{ES05} using the same changes of variables and replacing \cite[Thm. 2.6]{ES05} by Theorem~\ref{thm:shypersurface}.
\end{proof}

\begin{cor}\label{cor:CIconj}
Fix $d \geq 2$ and $k$ a field $k$ of characteristic $p > 2$. For $R$ a singular complete intersection of dimension $d$, $e_s(R) \geq e_s(R_d)$  for all $s > 0$.
\end{cor}
\begin{proof}
This follows the same proof as \cite[Thm. 4.6]{ES05}. Its standard to assume $k$ is uncountable and algebraically closed by base change. One simply replaces \cite[Thm. 3.2]{ES05} with Theorem~\ref{thm:shypersurface} noting that up to very delicate changes of variables and Weierstrass preparation arguments, none of which depends on the lengths calculating the Hilbert-Kunz multiplicity, one reduces to giving an $s$-analogue of \cite[Prop. 4.2]{ES05}. However, each step there again holds replacing again the proof of \cite[Thm. 2.6]{ES05} by the proof of Theorem~\ref{thm:shypersurface}.
\end{proof}

\subsection{An alternate lower bound in dimension $4$:} The main obstruction to pushing these results farther is that when $d \geq 4$, $R_d$ is no longer a toric ring, and hence it is not clear how to compute $e_s(R_d)$ for such $d$. Nonetheless, the computations for $e_s(R_d)$ for $d \leq 3$ suggest a general form for the function $s \mapsto e_s(R_d)$. In particular, consider the function 
\[\phi(s,d) := \begin{cases}  2-\frac{2\cH_{s-1}(d)}{\cH_s(d)}& \text{if } s\leq \frac{d+1}{2} \\ \frac{2\cH_{(d+1)/2}(d)-2\cH_{(d-1)/2}(d)}{\cH_s(d)} & \text{if } s\geq \frac{d+1}{2}.\end{cases}\]

We note in particular that $e_s(R_4) \neq \phi(s,4)$ as the former depends on the characteristic. However, we also have no direct comparison between these two functions. Either way, one can ask if $e_s(R) \geq \phi(s,d)$ for unmixed local rings of dimension $d$. We can establish that this holds for $d = 4$. 

\begin{lem}\label{difference descriptions} Let $d,r\in \NN_{\geq 1}$, and let $f_{d,r}:\RR_{>0}\to \RR$ be given by $f_{d,r}(s)=\cH_s(d)-r\cH_{s-1}(d)$.  There exists $t$ with $1\leq t\leq \frac{d+1}{2}$ such that $f_{d,r}$ is strictly increasing on the interval $(0,t)$ and strictly decreasing on  the interval $(t,d+1)$.
\end{lem}

\begin{proof} We proceed by induction on $d$.  If $d=1$, then
\[f_{d,r}(s)=\begin{cases} s &\text{if } s\leq 1  \\ 1-r(s-1) & \text{if } 1\leq s \leq 2 \\ 1-r & \text{if } s\geq 2\end{cases}\]
so the statement is proved by taking $t=1$.  Let $d\geq 2$ and suppose we know that there exists $0<u\leq \frac{d}{2}$ such that $f_{d-1,r}(s)$ is strictly increasing on $(0,u)$ and strictly decreasing on $(u,d)$.  Let $t\in (0,d+1)$ be a local extremum of $f_{d,r}$, which immediately implies that $t\geq 1$ since $f_{d,r}(s)$ is strictly increasing for $0<s<1$.  The function $f_{d,r}$ is differentiable with continuous derivative, and so $0=f'_{d,r}(t)=f_{d-1,r}(t)-f_{d-1,r}(t-1)$.   Since $f_{d-1,r}(s)$ is strictly increasing on $(0,u)$, strictly decreasing on $(u,d)$, and constant on $(d,\infty)$, there is a unique $t$ such that $1\leq t\leq d+1$ and $f_{d-1,r}(t)=f_{d-1,r}(t-1)$.  Therefore $t$ is the unique extremum of $f_{d,r}$ on $(0,d+1)$, and it must be a maximum since $f_{d,r}(s)$ is increasing for $s<1$.

To prove that $t$ must be no more than $\frac{d+1}{2}$, we proceed by induction on $r$.  If $r=1$ and $t$ is the local maximum for $f_{d,r}$, then
\[f_{d,1}(t)=\cH_t(d)-\cH_{t-1}(d)=1-\cH_{d-t}(d)-(1-\cH_{d-t+1}(d))=\cH_{d-t+1}(d)-\cH_{d-t}(d)=f_{d,1}(d-t+1).\]
Since $f_{d,1}$ has only one local maximum, $t$ must equal $d-t+1$, and so $t=\frac{d+1}{2}$.  Now suppose that the maximum value for $f_{d,r-1}$ occurs at a point $t\leq \frac{d+1}{2}$ for some $r\geq 2$.  Let $s,s'\in (0,d+1)$ such that $t<s<s'$.  In this case,
\[f_{d,r}(s')-f_{d,r}(s)=f_{d,r-1}(s')-f_{d,r-1}(s)-\cH_{s'}(d)+\cH_{s}(d)<0.\]
Therefore $f_{d,r}$ is decreasing on the interval $(t,d+1)$, and so its maximum must occur at a point less than or equal to $t$, which finishes the induction.
\end{proof}

\begin{thm} For singular Cohen-Macaulay unmxied local rings $R$ of dimension $4$, $e_s(R) \geq \phi(s,4) $. 
\end{thm}

\begin{proof} The previous theorem proved the statement for rings of dimension up to 3, so let $R$ be a singular ring of dimension $4$.  By Lemma \ref{difference descriptions},  $2\cH_s(4)-2\cH_s(4)\leq  2\cH_{5/2}(4)-2\cH_{3/2}(4)=\frac{115}{96}$ for all $s>0$.

If $e=2$, then taking $t=\min\{s,\frac{d+1}{2}\}$ in Theorem \ref{thm:MainLowerBound} gives us the statement we desire.  If $3\leq e\leq 10$, then $\frac{\sqrt[4]{5+e}}{\sqrt[8]{5+e}-1}\geq 2$, and so $e_s(R)\geq 2-\frac{2\cH_{s-1}(4)}{\cH_s(4)}$ for $s\leq 2$ by Lemma \ref{small s inequality}.  On the other hand, if $s\geq 2$, then
\[e_s(R)\geq \frac{e\cH_2(4)-e(e-1)\cH_1(4)}{\cH_s(4)}=\frac{\frac{13}{24}e-\frac{1}{24}e^2}{\cH_s(4)}\geq\frac{115}{96\cH_s(4)}\geq 2-\frac{2\cH_{s-1}(4)}{\cH_s(4)}.\]

If $11\leq e\leq 28$, then  $\frac{\sqrt[4]{5+e}}{\sqrt[8]{5+e}-1}\geq \frac{3}{2}$, and so $e_s(R)\geq 2-\frac{2\cH_{s-1}(4)}{\cH_s(4)}$ for $s\leq \frac{3}{2}$ by Lemma \ref{small s inequality}.  If $s\geq \frac{3}{2}$, then
\[e_s(R)\geq \frac{e\cH_{3/2}(4)-e(e-1)\cH_{1/2}(4)}{\cH_s(4)}=\frac{\frac{78}{384}e-\frac{1}{384}e^2}{\cH_s(4)}\geq\frac{115}{96\cH_s(4)}\geq 2-\frac{2\cH_{s-1}(4)}{\cH_s(4)} .\]

If $e\geq 29$, then
\[e_s(R)\geq \frac{e\cH_{1}(4)-e(e-1)\cH_{0}(4)}{\cH_s(4)}=\frac{\frac{1}{24}e}{\cH_s(4)}\geq\frac{115}{96\cH_s(4)} \geq 2-\frac{2\cH_{s-1}(4)}{\cH_s(4)}. \qedhere\]
\end{proof}

\section*{Appendix}

Here we flesh out the computation of Example~\ref{xmp:R3}. In particular for $R_3=k[[X,Y,Z,W]]/(XY-ZW)$, we need to calculate $e_s(R_3)$. For $s\leq 1$, we have by \cite[Cor. 3.7(i)]{Tay}, that $e_s(R_3)=e(R_3)=2$. 

Set $\fm=(X,Y,Z,W)$. We start by bounding the $F$-threshold of $\fm$ with respect to itself. It suffices to show for $e \in \NN$, $\fm^{2 p^e} \subset \fm^{[p^e]}$. Let $X^aY^bZ^cW^d \in \fm^ {2p^e}$.  Without loss of generality, we may assume that $a\leq b$ and $c\leq d$.  Since $a+b+c+d\geq 2p^e$, either $a+d\geq p^e$ or $b+c\geq p^e$.  In the first case, $X^aY^bZ^cW^d=Y^{b-a}Z^{a+c}W^{a+d}\in\fm^{[p^e]}$, and the other case is similar.  Therefore the $F$-threshold of $\fm$ with respect to $\fm$ is at most 2, and hence, by \cite[Cor. 3.7(ii)]{Tay}, if $s\geq 2$ then $e_s(R_3)=\frac{e_{HK}(R_3)}{\cH_s(R_3)}=\frac{4}{3\cH_s(3)}$.

Now assume $1\leq s\leq 2$.  Using \cite[Thm. 5.4]{Tay}, we may calculate $e_s(R_3)$ as a volume in $\RR^3$.  In particular, let $e_1,e_2,e_3$ be the standard basis vectors for $\RR^3$.  The ring $R_3$ is the affine semigroup ring of the cone $\sigma^\vee\subseteq \RR^3$ generated by $\{e_1,e_2,e_3,e_1+e_2-e_3\}$. Thus, the colength $h_s(R_3) = \lambda(R_3/ \fm^{ \lceil sq \rceil} + \fm^{[q]})$ is the volume of the set 
\[U=\left\{v\in \sigma^\vee \;|\;v\notin s\cdot \Hull \fm\cup (\Exp \fm+\sigma^\vee)\right\}.\]

To justify the example, it suffices to express this volume as an integral. 

\begin{thm}
We have $$\vol(U)=2\int_0^1\cH_{s-z}(2)-z^2\cH_{1-(2-s)/z}(2)dz.$$
\end{thm}

\begin{proof}
To show this, we determine a collection of inequalities on the coordinates which identify $v=(x,y,z)\in \sigma^\vee$. In particular, we claim that $(x,y,z)\in \sigma^\vee$ if and only if $x,y,x+z,y+z\geq 0$

To see this, note for any $v=(x,y,z)\in \sigma^\vee$ there exist $a_i\in \RR_{\geq 0}$ such that $v=a_1e_1+a_2e_2+a_3e_3+a_4(e_1+e_2-e_3)$, and so $x=a_1+a_4$, $y=a_2+a_4$, and $z=a_3-a_4$.  This  implies that $x,y,x+z,y+z\geq 0$, and further that $x+y+z=a_1+a_2+a_3+a_4$.  On the other hand, suppose  $x,y,x+z,y+z\geq 0$.  If $z\geq 0$, then we may take $a_1=x$, $a_2=y$, $a_3=z$, and $a_4=0$ to see that $(x,y,z)\in \sigma^\vee$.  If $z<0$, then we may take $a_1=x+z$, $a_2=y+z$, $a_3=0$, and $a_4=-z$ to realize the same.  Note this also gives a condition for membership in $s\cdot\Hull\fm =\{a_1e_1+a_2e_2+a_3e_3+a_4(e_1+e_2-e_3)\;|\;a_i\in \RR_{\geq 0}, \sum a_i\geq s\}$, in particular, 
\[v=(x,y,z)\in s\cdot\Hull \fm\Leftrightarrow x+y+z=a_1+a_2+a_3+a_4\geq s.\]

It suffices now to analyze $\Exp \fm+\sigma^\vee$. To this end, we claim $v=(x,y,z)\in e_1+\sigma^\vee$ if and only if $(x-1,y,z)\in \sigma^\vee$, and so this occurs precisely when $x-1,y,x+z-1,y+z\geq 0$.  Similar arguments hold for $e_2+\sigma^\vee$ and $e_3+\sigma^\vee$.  For $v\in(e_1+e_2-e_3)+\sigma^\vee$, we wish to know when $v-(e_1+e_2-e_3)=(x-1,y-1,z+1)\in \sigma^\vee$, which occurs precisely when $x-1,y-1,x+z,y+z\geq 0$.  Thus we have that
\[v=(x,y,z)\in \Exp\fm+\sigma^\vee \Leftrightarrow \left\{\begin{array}{l} x-1,y,x+z-1,y+z\geq 0, \text{ or }\\ x,y-1,x+z,y+z-1\geq 0,\text{ or }\\ x,y,x+z-1,y+z-1\geq 0,\text{ or }\\x-1,y-1,x+z,y+z\geq 0\end{array}\right\}\] It is now routine to establish that $U$ consists of the set of points $v = (x,y,z)$ such that 
\begin{enumerate}
\item $x,y,x+z,y+z\geq 0$, and 
\item $x+y+z\leq s$, and 
\item either $x<1$ and $y+z<1$, or $y<1$ and $x+z<1$.
\end{enumerate}

% Therefore, the set $U$ consists of those points $v=(x,y,z)$ such that:
% \begin{enumerate}
% \item $x,y,x+z,y+z\geq 0$, and 
% \item $x+y+z< s$, and 
% \item $x<1$ or {\color{red} $y<0$} or $x+z<1$ or ${\color{red}y+z<0}$, and
% \item ${\color{red}x<0}$ or $y<1$ or ${\color{red}x+z<0}$ or $y+z<1$, and
% \item ${\color{red}x<0}$ or ${\color{red}y<0}$ or $x+z<1$ or $y+z<1$, and
% \item $x<1$ or $y<1$ or ${\color{red}x+z<0}$ or ${\color{red}y+z<0}$.
% \end{enumerate}
% The inequalities in red cannot occur if the first condition is satisfied, so we can simplify the conditions to
% \begin{enumerate}
% \item $x,y,x+z,y+z\geq 0$, and 
% \item $x+y+z\leq s$, and 
% \item $x<1$  or $x+z<1$, and
% \item $y<1$  or $y+z<1$, and
% \item $x+z<1$ or $y+z<1$, and
% \item $x<1$ or $y<1$.
% \end{enumerate}
% Conditions 3-6 above can be collapsed, so that a new equivalent system of conditions is
% \begin{enumerate}
% \item $x,y,x+z,y+z\geq 0$, and 
% \item $x+y+z\leq s$, and 
% \item either $x<1$ and $y+z<1$, or $y<1$ and $x+z<1$.
% \end{enumerate}

To calculate the volume of $U$, we integrate over the one of the coordinates. For a fixed nonnegative value of $z$, the points $(x,y)$ such that $(x,y,z)\in U$ are the points such that
\begin{enumerate}
\item $0\leq x,y\leq 1$, and
\item $x+y\leq s-z$, and
\item it is not the case that both $x\geq 1-z$ and $y\geq 1-z$.
\end{enumerate}
From condition (3) we see that we may assume that $z\leq 1$, and we do so.  The volume of points satisfying the first two conditions is exactly $\cH_{s-z}(2)$.  The points that satisfy the first two conditions but fail the second can be written in the form $(1-z+\alpha,1-z+\beta)$, where $0\leq \alpha,\beta\leq z$ and $\alpha +\beta \leq s-z-(2-2z)=z-(2-s)$.  The volume of these points is $z^2\cH_{1-(2-s)/z}(2)$. Now fixing a negative value of $z$, the points $(x,y)$ such that $(x,y,z)\in U$ are the points such that
\begin{enumerate}
\item $-z\leq x,y\leq 1-z$, and
\item $x+y\leq s-z$, and
\item it is not the case that both $x\geq 1$ and $y\geq 1$.
\end{enumerate}
Setting $x'=x+z$ and $y'=y+z$, the volume of the points $(x,y)$ satisfying the above conditions is the  same as the volume of the points $(x',y')$ satisfying
\begin{enumerate}
\item $0\leq x,y\leq 1$, and
\item $x+y\leq s+z$, and
\item it is not the case that both $x'\geq 1+z$ and $y'\geq 1+z$.
\end{enumerate}
Thus these points have exactly the same volume as those in the case where $z$ is nonnegative. 

\end{proof}

% Therefore,
% \begin{align*}
% h_s(R_3)=\vol(U)&=2\int_0^1\cH_{s-z}(2)-z^2\cH_{1-(2-s)/z}(2)dz=2\cH_s(3)-2\int_{2-s}^1\frac{1}{2}z^2\left(1-\frac{2-s}{z}\right)^2dz\\
% &=2\cH_s(3)-\int_{2-s}^1(z-(2-s))^2dz=2\cH_s(3)-\frac{(s-1)^3}{3}=2\cH_s(3)-2\cH_{s-1}(3).
% \end{align*}

\end{document}